\documentclass{amsart}

\usepackage{amsfonts,amssymb,amsmath,amscd,amstext}
\usepackage[colorlinks=true,linkcolor=blue,citecolor=blue]{hyperref}
\usepackage[utf8]{inputenc}
\usepackage{microtype}
\usepackage{graphicx}
\usepackage{changes}
\usepackage{esint}
\usepackage{pdfpages}
\graphicspath{{pictures/}}

\renewcommand{\leq}{\leqslant}
\renewcommand{\geq}{\geqslant}
\renewcommand{\le}{\leqslant}
\renewcommand{\ge}{\geqslant}

\newcommand{\rr}{{\mathbb{R}}} 
\newcommand{\ag}{{\mathfrak{g}}}

\newcommand{\hhh}{{\mathcal{H}}}
\newcommand{\la}{\lambda}

\newcommand{\Om}{\Omega}

\newcommand{\escpr}[1]{\langle#1\rangle}
\newcommand{\norm}[1]{|| #1 ||}

\DeclareMathOperator{\divv}{div}
\DeclareMathOperator{\intt}{int}

\newtheorem{theorem}{Theorem}[section]
\newtheorem{proposition}[theorem]{Proposition}
\newtheorem{lemma}[theorem]{Lemma}
\newtheorem{corollary}[theorem]{Corollary}

\theoremstyle{definition}

\newtheorem{remark}[theorem]{Remark}

\theoremstyle{remark}

\numberwithin{equation}{section}

\definechangesauthor[name=Julian, color=blue]{J}

\begin{document}

\title{Existence of isoperimetric regions in sub-Finsler nilpotent groups}

\author[J.~Pozuelo]{Juli\'an Pozuelo} 

\address{Departamento de Geometr\'{\i}a y Topolog\'{\i}a \\
	Universidad de Granada \\ E--18071 Granada \\ Espa\~na}

\email{pozuelo@ugr.es}

\date{\today}

\thanks{The author would like to thank his Ph.D. advisor Manuel Ritoré for sugesting the problem and his help.}

\thanks{The author has been supported by MEC-Feder
	grant MTM2017-84851-C2-1-P, Junta de Andalucía grant A-FQM-441-UGR18, MSCA GHAIA, and Research Unit MNat UCE-PP2017-3} 

\begin{abstract}
 We consider a nilpotent Lie group with a bracket-generating distribution $\hhh$ and an asymmetric left-invariant norm $\|\cdot\|_K$ induced by a convex body $K\subseteq\hhh$ containing $0$ in its interior. In this paper we prove the existence of minimizers of the perimeter functional $P_K$ associated to $\|\cdot\|_K$ under a volume (Haar measure) constraint. Our result generalizes the one of Leonardi and Rigot for Carnot groups.
\end{abstract}

\maketitle

\thispagestyle{empty}

\bibliographystyle{abbrv} 


\section{Introduction}

A well-known problem in Euclidean space is the isoperimetric problem of minimizing perimeter under a volume constraint. To consider this problem, we need a notion of volume $|E|$ of an open set  $E\subseteq \rr^d$, the Lebesgue measure,  and a notion of perimeter, given by
\[
P(E)=\sup\bigg\{\int_E\divv U\,dx: U\in\mathfrak{X}_0^1(\rr^d), \norm{U}_{\infty}\le 1\bigg\}.
\]
When $P(E)<+\infty$, $E$ is said to be a finite perimeter set. It is natural to ask for existence of sets $B\subseteq \rr^d$ such that
\[
P(E)\geq P(B),
\]
for all finite perimeter sets $E$ with $|E|=|B|$.
%

Another notion of perimeter is the Finsler perimeter: taking an asymmetric norm in $\rr^d$, $\|\cdot\|_K$, we can consider the perimeter as
\[
P_K(E)=\sup\bigg\{\int_E\divv U\,dx: U\in\mathfrak{X}_0^1(\rr^d), \norm{U}_{K,\infty}\le 1\bigg\}.
\]
The sets minimizing the perimeter for its volume are usually referred to as the \emph{Wulff shapes} and were described by the crystallographer G. Wulff in 1895. A first mathematical proof was given by Dinghas \cite{zbMATH03045073}. Other versions of Wulff's results were given by Busemann \cite{MR31762}, Taylor \cite{MR493671}, Fonseca \cite{MR1116536} and Fonseca and Müller  \cite{MR1130601}; see also Gardner \cite{MR1898210} and Burago and Zalgaller \cite{MR936419}.

In the recent years there has been an increase interest in the isoperimetric problem in the context on metric measure spaces. The privilege position of Carnot groups within geometric measure theory is revealed by the characterization as the only metric spaces that are
\begin{enumerate}
	\item locally compact,
	\item geodesic,
	\item isometrically homogeneous, and
	\item self-similar (i.e. admitting a dilation).
\end{enumerate}
This characterization can be seen as Theorem 1.1 in \cite{MR3283670}. Removing the self-similarity condition, sub-Finsler nilpotent groups acquire relevant importance.
In \cite{MR676380}, Pansu proved an isoperimetric inequality in the Heisenberg group and conjectured that some spheres are the isoperimetric regions. In \cite{MR1404326}, Garofalo and Nhieu proved an isoperimetric inequality in $\rr^n$ with a family of vector fields satisfying a Hörmander condition and supposing a doubling property, a Poincaré inequality and that $\rr^n$ with the sub-Riemannian distance is a complete length-space.   There is a number of partial results on the sub-Riemannian isoperimetric problem in the Heisenberg group, that can be found in \cite{MR2435652,MR2898770,MR2402213,MR2548252}.

Considering an asymmetric left-invariant norm in the horizontal distribution of the Heisenberg group, we obtain the sub-Finsler Heisenberg group. In this group the Minkowski content is considered in \cite{snchez2017subfinsler}, while two partial results are generalized in  \cite{monti-finsler,ritore-pozuelo}.  

For a different notion of perimeter of a submanifold of fixed degree immersed in a graded manifold that depends on the degree, see \cite{ritore-gimmy} (see also \cite{MR2414951}). 

In Riemannian geometry, a very general result of existence was given by Morgan in \cite{Morgan}, when a subgroup of the isometry group of $M$ acts cocompactly on $M$. Different results are obtained in \cite{MR3514556,MR3215337,isop}.

In contact sub-Riemannian manifolds, Galli and Ritoré  proved in \cite{MR2979606} an analog result to the Morgan's existence result. For that, they followed Morgan's structure:  they pick a minimizing sequence of sets of volume $v$ whose perimeters approach the infimum of the perimeters of sets of volume $v$. This sequence can be splitted into two subsequences. The first subsequence is converging to a set, and it is proved that is isoperimetric for its volume and bounded. Nevertheless, it might be a loss of mass at infinity. In this case, they use isometries to translate the second subsequence, which is diverging, to recover some of the lost volume. An essential point is that they always recover a fixed fraction of the volume.

Leonardi and Rigot proved  in \cite{MR2000099} existence of isoperimetric regions in Carnot groups, that is, nilpotent Lie group which is equipped with a family of dilations, together with the properties of Ahlfors-regularity and Condition B of its boundary. The dilations plays a key role, since they allow to prove that the isoperimetric profile is a concave function, and that there is no loss of mass at infinity.

In this paper we give a proof of existence of isoperimetric regions for any volume in a nilpotent group with a set of left-invariant vector fields $X$ satisfying a Hörmander condition, that is, the Lie bracket generating condition, and an asymmetric left-invariant norm $\|\cdot\|_K$,  without the assumption of been equipped with a family of dilations (Theorem \ref{thm:existence}). We shall adapt the arguments of Galli and Ritoré \cite{MR2979606} to prove existence of isoperimetric regions (see also \cite{leonardi-ritore}). As in \cite{MR2979606}, the main difficulty is to prove a Deformation Lemma (Lemma \ref{lem:def}) which will allow us to increase the volume of any finite perimeter set while modifying the perimeter in a controlled way. Moreover, we obtain that the isoperimetric profile is non-decreasing (Proposition \ref{lem:non}) and, from an uniform version of the Deformation Lemma and the existence of isoperimetric regions, we will deduce that the isoperimetric profile is locally Lipschitz and sub-additive in in Proposition \ref{prop:cont} and Corollary \ref{cor:sub-add} respectively. We shall also extend the properties obtained in Carnot groups by Leonardi and Rigot in \cite{MR2000099}, that isoperimetric regions are bounded and its topological boundary and the essential boundary coincide.

This paper is organized as follows. In Section 2, we establish notation and give some background on sub-Finsler nilpotent groups and the notion of $K$-perimeter.
In Section 3, we study some properties of the isoperimetric regions such as that they are open up to a nullset (Corollary \ref{cor:open}), bounded (Theorem \ref{thm:bound}) and its essential and topological boundaries coincide (Theorem \ref{thm:boundary}), and prove in Proposition \ref{lem:non} that the isoperimetric profile is non-decreasing. Moreover, we prove a Deformation Lemma, Lemma \ref{lem:def}, for a finite $K$-perimeter set in a nilpotent group. In Section 4, we prove in Theorem \ref{thm:existence} the main result, the existence of isoperimetric regions and deduce that the isoperimetric profile is a locally Lipschitz and sub-additive function in Proposition \ref{prop:cont} and Corollary \ref{cor:sub-add}.

\section{Preliminaries}

\subsection{Nilpotent groups}

We recall some results on nilpotent groups. For a quite complete description of nilpotent Lie groups the reader is referred to Section 1.13 in \cite{MR1920389}.

Let $\ag$ be a Lie algebra. We define recursively $\ag_0=\ag$, $\ag_{i+1}=[\ag,\ag_i]=\mbox{span}\{[X,Y] : X\in\ag,Y\in\ag_i \}$. The decreasing series		
$$\ag=\ag_0\supseteq\ag_1\supseteq\ag_2\supseteq\ldots$$
is called the \emph{lower central series of $\ag$}. If $\ag_r=0$ and $\ag_{r-1}\neq0$ for some $r$, we say that $\ag$ is \emph{nilpotent} of step $r$. A connected and simply connected Lie group is said to be \emph{nilpotent} if its Lie algebra is nilpotent.

The following Lemma will be used in the sequel.
\begin{lemma}\label{seq}
	Let $\ag$ be a nilpotent Lie algebra. Then there exists a basis $\{Y_1,\ldots,Y_d\}$ of $\ag$ such that 
	\begin{enumerate}
		\item for each $1\leq n\leq d$, $\mathfrak{h}_n=\mbox{span}\{Y_{d-n+1},\ldots,Y_d\}$ is an ideal of $\ag$,
		\item  for each $0\leq i\leq r-1$, $\mathfrak{h}_{n_i}=\ag_i$.
	\end{enumerate}	
\end{lemma}

A basis verifying this is called a \emph{Malcev basis}. This construction is adapted from \S~1.2 in \cite{MR1070979}. An important feature of this basis is that $Y_d\in\ag_{r-1}\subseteq\text{Center}(\ag)$, that is, commutes with any $U$ in $\ag$. Fixed a Malcev basis, the exponential map provides a diffeomorphism between $\rr^d$ and $G$, given by the map 
$$x=(x_1,\ldots,x_d)\mapsto \exp(x_1Y_1+...+x_dY_d). $$
This result can be found as Theorem 1.127 in \cite{MR1920389}. The inverse of this map provides coordinates called \emph{canonical coordinates of the first kind}. The group product can be recovered by the Hausdorff-Campbell-Baker formula as
\begin{equation*}
(x_1,\ldots,x_d)\cdot (y_1,\ldots,y_d)=\exp^{-1}(X+Y+\frac{1}{2}[X,Y]+\frac{1}{12}[X,[X,Y]]+\ldots),
\end{equation*}
where $X=\sum x_iY_i$, $Y=\sum y_i Y_i$, $(x_1,\ldots,x_d)=\exp(X)$ and $(y_1,\ldots,y_d)=\exp(Y)$. In particular, 
\begin{equation}\label{eq:tras}
x\cdot (0,\ldots,0,y_d)=x+ (0,\ldots,0,y_d).
\end{equation}

 The Lebesgue measure of $\rr^d$ will be denoted as $|\cdot|$ and, as we can see Theorem 1.2.10 in \cite{MR1070979}, it coincides with the Haar measure on $\rr^d$ with this product.

 From now on, we shall denote a nilpotent group as $(\rr^d,\cdot)$.

Given a nilpotent group $(\rr^d,\cdot)$ and a system of linearly-independent left-invariant vector fields $X=\{X_1,\ldots, X_{k}\}$, we define the distributions
\begin{equation*}
\hhh^0:=\text{span}(X) \qquad \hhh^n:=\hhh^{n-1}+[\hhh^{n-1},\hhh^0]
\end{equation*}
 we say that $X$ is a \emph{Lie bracket generating system} if $\hhh^s(0)=\ag$ for some $s$. There is a left-invariant distance $d_X$ in $\rr^d$ associated to $X$ (see \cite{MR793239}). 
%
The distribution $\hhh^0$ is called the \emph{horizontal distribution}. A vector field $U$ is said to be \emph{horizontal} if $U(x)\in \hhh^0_x$ for all $x$ in $\rr^d$.  
 
  We consider a left-invariant Riemannian metric $g:=\langle\cdot,\cdot\rangle$ forming $X$ an orthonormal basis of $\hhh^0$, and making orthogonal the subbundles $\hhh^0$ and $V$, where $V$ is a complementary subbundle of $\hhh^0$.

Given a nilpotent group $(\rr^d,\cdot)$ and a Lie bracket generating system $X$, the \emph{dimension at infinity}, $D$, and the \emph{local dimension}, $l$, of $\rr^d$ and $X$, are defined as
\begin{equation*}
D=\sum i m_i, \qquad l=\sum i m'_i.
\end{equation*}
where $m_i:=\dim(\ag_{i-1})-\dim(\ag_{i})$, $m'_i:=\dim(\hhh^{i-1})-\dim(\hhh^{i-2})$ and $m_1:=\dim(\hhh^0)$.

\begin{remark}	
	It is clear that $D\geq l$. Moreover, $D=l$ if and only if $(\rr^d,\cdot)$ is stratifiable (that is, $\ag$ admits a direct-sum decomposition $\ag=V_0\oplus v_1\ldots\oplus V_r$ such that $V_s\neq\{0\}$ and $[V_1,V_j] = V_{1+j}$, where $1<j<s$ and $V_{s+1}=\{0\}$) and $X$ generates a stratification on $\ag$.
\end{remark}
\begin{remark}
	 In a stratifiable group, not every subspace $V_0$ such that $\ag= V_0\oplus\ag_1$ generates a stratification. Moreover, not every nilpotent group is stratifiable (see Example 1.8 and 1.9 \cite{MR3742567}).
\end{remark}	

The following result can be seen in Section IV.5 of \cite{MR1218884}.
\begin{theorem}[Proposition IV.5.6 and Proposition IV.5.7 in \cite{MR1218884}]\label{thm:balls}
	Let $(\rr^d)$ be a nilpotent group with $X$ a bracket-generating system and $D$ and $l$ the dimension at infinity and the local dimension respectively. There exist positive constants $\alpha$ and $\beta$ such that
	\begin{equation}\label{eq:ball}
	\begin{split}	
		 \alpha^{-1}t^l\leq &| B(0,t)|\leq \alpha t^l \quad 0\leq t\leq 1,\\
		 \beta^{-1}t^D \leq&|B(0,t)|\leq \beta t^D \quad t>1.
	\end{split}
	\end{equation}
\end{theorem}	
In particular, we obtain the following inequality.
	\begin{equation}\label{in:doub}
	\mu(B(x,s))\leq C\Big( \frac{s}{r}\Big)^{l}\mu(B(x,r)),
	\end{equation}
	where $x\in \rr^d$, $0<r\leq s\leq 1$.
%
%
%
\subsection{sub-Finsler norms}
We follow the approach developed in \cite{ritore-pozuelo}.

	Let $V$ be a vector space. We say that $|\cdot|:V\to \rr^+$ is a norm if it verifies
\begin{enumerate}
	\item $|x|=0\Leftrightarrow x=0$.
	\item $|sx|=s|x|$  $\forall s>0$ and $\forall x\in V$.
	\item $|x+y|\leq|x|+|y|$ $\forall x,y\in V$.
\end{enumerate}
We stress the fact that we are not assuming the symmetry property $\norm{-v}=\norm{v}$.

Associated to a given a norm $||\cdot||$ in $V$ we have the set $F=\{u\in V:||u||\le 1\}$, which is compact, convex and includes $0$ in its interior. Reciprocally, given a compact convex set $K$ with $0\in\intt(K)$, the function 
$$||u||_K=\inf\{\la\ge 0:  u\in\la K\}$$
 defines a norm in $V$ so that $F=\{u\in V:||u||_K\le 1\}$.

Given a norm $\norm{\cdot}$ and an scalar product $\escpr{\cdot,\cdot}$ in $V$, we consider its dual norm $\norm{\cdot}_*$ of $||\cdot||$ with respect to $\escpr{\cdot,\cdot}$ defined by
\[
\norm{u}_*=\sup_{\norm{v}\le 1}\escpr{u,v}.
\]
The dual norm is the support function $h$ of the unit ball $K=\{u\in V: \norm{u}\le 1\}$ with respect to the scalar product $\escpr{\cdot,\cdot}$.

Given $u\in V$, the compactness of the unit ball of $\norm{\cdot}$ and the continuity of $\norm{\cdot}$ imply the existence of $u_0\in V$ satisfying the equality $\norm{u}_*=\escpr{u,u_0}$. Moreover, it can be easily checked that $\norm{u_0}=1$.
We shall define $\Pi(u)$ as the  set of vectors satisfying  $\norm{\pi(u)}=1$ and
\begin{equation}\label{eq:proy}
h(u)=\norm{u}_*=\escpr{u,\pi(u)}.
\end{equation}
for any $\pi(u)\in \Pi(u)$. 
%
%
%
%

Now we move to a nilpotent group with a bracket generating system $X$. Given a convex set $K\subseteq \hhh^0_0$ containing the origin,  we construct a norm $\norm{\cdot}$ in $\hhh^0$ extending the norm $\|\cdot\|_K$ in $\hhh^0_0$ by left-invariance. For that, we take a horizontal vector at $p\in\rr^d$, $v=\sum_i a_iX_i(p)$, by means of the formula
\begin{equation}
\norm{\sum_i a_iX_i(p)}_p=\norm{\sum_i a_iX_i(0)}_K.
\end{equation}
 
We say that $(\rr^d,\cdot,X,K)$ is a sub-Finsler nilpotent group.

\subsection{Sub-Finsler perimeter}
From now on, we shall identify a nilpotent group $G$ with $(\rr^d,\cdot)$. Let $E\subset \rr^d$ be a measurable set, $\norm{\cdot}_K$ the left-invariant norm associated to a convex body $K\subset \hhh^0_0$ so that $0\in\intt(K)$, and $\Om\subset \rr^d$ an open subset. We say that $E$ has locally finite $K$-perimeter in $\Om$ if we have
\[
P_K(E;\Om)=\sup\bigg\{\int_E\divv(U)\,dx: U\in\hhh_0^1(\Om), \norm{U}_{K,\infty}\le 1\bigg\}<+\infty.
\]
In this expression, $\hhh_0^1(\Om)$ is the space of horizontal vector fields of class $C^1$ with compact support in $\Om$, $\norm{U}_{K,\infty}=\sup_{p\in V} \norm{U_p}_K$  and the divergence is computed with respect to the metric $g$. The integral is computed with respect to the Lebesgue measure $dx$ on $\rr^d$. In case that $\Omega= G$ we write $P_{K}(E):=P_{K}(E;\Omega)$. When $K$ is the Euclidean ball in $\hhh^0_0$, we recover the sub-Riemannian perimeter, and drop the subindex $K$.

Let $K\subset \hhh^0_0$ be a bounded convex body containing $0$ in its interior. Then there exists a constant $C_1>1$ such that
\[
C_1^{-1} \norm{x}\le \norm{x}_K\le C_1 \norm{x},\quad \text{for all }x\in\rr^d,
\]
where $\|\cdot\|$ is the Euclidean norm in $\rr^d$. Let $E\subset\rr^d$ be a measurable set and $\Om\subset\rr^d$ an open set. Take $U\in\hhh_0^1(\Om)$ a vector field with $\norm{U}_{K,\infty}\le 1$. Hence $\norm{C_1^{-1} U}\le \norm{U}_{K}\le 1$ and
\begin{equation*}
	\int_E \divv(U)dx=C_1\int_E \divv(C_1^{-1} U)\,dx\le C_1 P(E;\Om),
\end{equation*}
Taking supremum over the set of $C^1$ horizontal vector fields with compact support in $\Om$ and $\norm{\cdot}_K\le 1$, we get $ P_{K}(E;\Om)\le C_1 P(E;\Om)$. In a similar way we get the inequality $C_1^{-1} P(E;\Om)\le  P_{K}(E;\Om)$, so that we have
\begin{equation}\label{eq:abscont}
C_1^{-1}  P(E;\Om)\leq  P_{K}(E;\Om)\leq C_1 P(E;\Om).
\end{equation}
As a consequence,  $E$ has finite $K$-perimeter if and only if it has finite (sub-Riemannian) perimeter. Hence all known results in the standard case apply to the sub-Finsler perimeter.

The proof of the Riesz Representation Theorem can be adapted, see \S~2.4 in \cite{ritore-pozuelo}, to obtain the existence of a Radon measure $|\partial E|_K$ on $\rr^d$ and a $|\partial E|_K$-measurable horizontal vector field $\nu_K$ in $\rr^d$ so that 
\begin{equation}\label{eq:div}
\int_E \divv(U)dx=\int_{\rr^d}\langle U,\nu_K\rangle d|\partial E|_K,
\end{equation}
where $U$ is a $C^1_0$ horizontal vector field.
Moreover, the $K$-perimeter can be represented by
\begin{equation}\label{reper}
P_K(E;\Omega)=\int_\Omega \|\nu_K\|_{K,*}d|\partial E|_K .
\end{equation}

 From the definition and that the Lebesgue measure is invariant by left-translations it follows that, for any $x\in \rr^d$,
\[
P_K(\ell_x\cdot E)=P_K(E).
\]
In particular, using Equation \eqref{eq:tras}, we get that the $K$-perimeter is invariant by Euclidean translations in the $d$ component, that is,
\[
P_K(E+(0,\ldots,0,a))=P_K(E).
\]

The following decomposition of the $K$-perimeter will be used exhaustively:
\begin{equation}\label{des:decom}
P_K(E)\geq P_K(E\cap B)+P_K(E\setminus B)-2C_1P(E\cap B;\partial B),
\end{equation}
where $B$ is any sub-Riemannian ball in $\rr^d$. Indeed,  as $\{B, \rr^d\setminus\bar{B}, \partial B\}$ is a partition of $\rr^d$, the Representation formula \eqref{reper} gives
\[
P_K(E)=P_K(E;B)+P_K(E;\rr^d\setminus \bar{B})+P_K(E;\partial B).
\]
Using Inequality \eqref{eq:abscont}, we get
\[
P_K(E;B)\geq P_K(E\cap B)-P_K(E\cap B;\partial B)\geq P_K(E\cap B)-C_1 P(E\cap B;\partial B).
\]
Similarly, $P_K(E;\rr^d\setminus \bar{B})\geq P_K(E\setminus B)-C_1 P(E\cap B;\partial B)$.

 The following relation between the sub-Riemannian perimeter and the derivative of the volume can be found as Lemma 3.5 in \cite{MR1823840}.

\begin{lemma}
	Let $(\rr^d,\cdot,X,K)$ be a sub-Finsler nilpotent group. Let $F\subseteq \rr^d$ be a finite (sub-Riemannian) perimeter set and $B_r$ the sub-Riemannian ball of radius $r$ centered in $0$. Then for a.e. $r>0$, we have
	\begin{equation}\label{eq:volper}
	\max\{P(F\cap B_r;\partial B_r),P(F\backslash B_r;\partial B_r)\}\leq - \frac{d}{ds}\Big|_{s=r}|F\setminus B_{s}|.
	\end{equation}
\end{lemma}	
%

Given a measurable set $F$, the density of $F$ at $x$ is given by
\[
\lim_{t\to 0}\frac{|F\cap B(x,r)|}{|B(x,r)|}.
\]
The \emph{essential boundary} of $F$ is the set of points where the density of $F$ is neither $0$ nor $1$.
%

\subsection{Isoperimetric inequality for small volumes}
A fundamental tool in a metric measure space $(X,d,\mu)$ is the existence of a $(1,1)$-Poincaré Inequality, that is, the existence of constants $C\geq0$ and $\lambda\geq1$ such that
\begin{equation}\label{des:Poincare}
\int_{B(x,r)}|f-f_{x,r}| d\mu\leq Cr\int_{B(x,\lambda r)}|\nabla f|d\mu,
\end{equation}
 for all $f$ is a locally Lipschitz function, where $B(x,r)$ is the metric ball of center $x$ and radius $r$, $f_{x,r}:=1/|B(x,r)|\int_{B(x,r)}fd\mu$ and $|\nabla f|$ is an upper gradient of $f$ in the sense of Heinonen and Koskela \cite{MR1654771}. In the context of connected Lie groups with polynomial volume growth, a $(1,1)$-Poincaré inequality was proven by Varopoulos in \cite{MR892879}, and in $\rr^d$ with a Lie bracket generating system by Jerison in \cite{MR850547}. 
 As stated by Hajlasz and Koskela in Theorem 5.1 and 9.7 in \cite{MR1683160}, the $(1,1)$-Poincaré inequality \eqref{des:Poincare} together with Inequality \eqref{in:doub} implies the following Sobolev inequality.
 \begin{equation}\label{des:Sobolev}
 \Big(\fint_{B(x,r)}|f-f_{x,r}|^{l/(l-1)} d\mu\Big)^{(l-1)/l}\leq \tilde {C}r\fint_{B(x, r)}|\nabla f|d\mu.
 \end{equation}
  From Inequality \eqref{des:Sobolev} and \eqref{eq:abscont}, the relative isoperimetric inequality easily follows (see also Theorem 1.18 in \cite{MR1404326}).
  \begin{theorem}[Relative Isoperimetric inequality]\label{thm:relin}
 	Let $(\rr^d,\cdot,X,K)$ be a sub-Finsler nilpotent group with local dimension $l$. There exists $C_2>0$ such that if $F\subset \rr^d$ is any finite perimeter, then
 	\begin{equation}\label{eq:relin}
 	C\min\{|F\cap B(x,r)|,|B(x,r)\setminus F| \}^{(l-1)/l}\leq P_K(F),
 	\end{equation}
 	for any $x\in \rr^d$ and $0<r\leq1$.
 \end{theorem}
 
 From a classical covering argument, we obtain the following isoperimetric inequality for small volumes. The proof follows identically as in Lemma 3.10 of \cite{MR2979606}.
  \begin{theorem}[Isoperimetric inequality for small volumes]\label{thm:in}
 	Let $(\rr^d,\cdot,X,K)$ be a sub-Finsler nilpotent group with local dimension $l$. There exists $C_2>0$ and $v_0>0$ such that if $F\subseteq \rr^d$ is any finite perimeter set and $|F|<v_0$, then
 	\begin{equation}\label{eq:in}
 	C_2|F|^{l-1/l}\leq P_K(F).
 	\end{equation}
 \end{theorem}


\section{Properties of the Isoperimetric regions}

Throughout this section, $(\rr^d,\cdot,X,K)$ denote a nilpotent group with $X$ a Lie bracket generating system, $K$ denote a convex body in $\hhh^0_0$ containing $0$ in its interior, $B(x,r)$ the sub-Riemannian ball centered in $x$ of radius $r>0$. 
We shall see that a isoperimetric region $E$ is open up to a nullset and its topological and essential boundary coincide, using the arguments developed in \cite{MR2000099,MR1625982}, and the boundedness. Moreover, we shall prove a Deformation Lemma.

The isoperimetric profile is defined as
\[
I_K(v):=\inf \{P_K(E):  E\subseteq\rr^d \textit{ is a finite perimeter set and } |E|=v  \}.
\]

%
%
%
%
\begin{proposition}\label{lem:non}
Let $(\rr^d,\cdot,X,K)$ be a sub-Finsler nilpotent group. The isoperimetric profile is non-decreasing.
\end{proposition}	

\begin{proof}
Since $X$ is a bracket generating system, $X$ spans $\ag\setminus\ag_1$ and therefore $X\setminus\ag_1\neq \emptyset$. Let $X_1\in X\setminus\ag_1$ and $\mathcal{D}$ be the distribution orthogonal to $X_1$. Then $\mathcal{D}$ is involutive and by Frobenius' Theorem, there exists a hypersurface $S$ passing through $0$ with $ TS=\mathcal{D}$. Moreover, $S$ is orientable. Let $S^+$ and $S^-$ the open regions in $\rr^d$ with boundary $S$ and horizontal normal vectors $X_1$ and $-X_1$ respectively.
	
Fix $v>w>0$ and let $E_n\subseteq \rr^d$ such that $|E_n|=v$ and $P_K(E_n)=I_K(v)+\frac{1}{n}$. 
Let $p_n$ be such that $|\ell_{p_n}E_n\cap S^+|=w$. By abuse of notation, we will write $E_n$ and $E_n^-$ for $\ell_{p_n}E_n$ and $\ell_{p_n}E_n\cap S^-$ respectively.
Let $U\in\Pi_K(X_1)$ be a projection over $K$, that is, satisfying Equation \eqref{eq:proy}. Applying \eqref{eq:div} to $E_n^-$ and $U$, we get
\begin{equation*}
\int_{E_n^-}\divv U dx=-\int_{E_n\cap S}\langle\pi_K(X_1) ,X_1\rangle d|\partial E_n^-|_K +\int_{ E_n^-}\langle U,\nu\rangle d|\partial E_n^-|_K. 
\end{equation*}
Since $U$ has constant coordinates in $X$, $\divv U=0$. Therefore, from \eqref{reper} we get
\begin{equation}\label{eq:cal}
P_K(E_n\cap S^+;S)=\int_{ E_n^-}\langle U,\nu\rangle d|\partial E_n^-|_K\leq P_K(E_n;S^-). 
\end{equation}
Adding $P_K(E_n;S^+)$ to both sides of Equation \eqref{eq:cal}, we get
\begin{equation*}
I_K(w)\leq P_K(E_n\cap S^+)\leq P_K(E_n)=I_K(v)+\frac{1}{n}.\qedhere
\end{equation*}
\end{proof}

We shall need the following Lemma proven in \cite{MR2000099} for Carnot groups.
\begin{lemma}\label{lem:ext}
	Let $E$ be an isoperimetric region, $x\in\rr^d$ and $0<r\leq1$. Then there exists $\varepsilon>0$ such that  if $r^{-l}| B(x,r)\setminus E|\leq \varepsilon$, then
\[
| B(x,r/2)\setminus E|=0,
\]	
where is $B(x,r)$ the sub-Riemannian ball centered in $x$ of radius $r>0$.
\end{lemma}
\begin{proof}
Let $|E|=v$. Suppose that $r^{-l}|B(x,r)\setminus E|\leq\varepsilon$. Let $t>0$, $B:=B(x,t)$, $E_t:=E\cup B$ and $m(t):=|B(x,t)\setminus E|$. Then
\begin{equation}\label{eq:1}
\begin{split}
P_K(E_t)&=P_K(E_t;\setminus\bar{B})+P_K(E\cup B;\partial B)\\
&\leq P_K(E)-P_K(E;B)+P_K(E\cup B;\partial B).
\end{split}
\end{equation}
On the other hand, using Equation \eqref{eq:volper} and that $-|E^c\setminus B|'=m'(t)$, we get
\begin{equation}\label{eq:2}
P_K(E\cup B;\partial B)\leq C_1 P(E\cup B;\partial B)= C_1 P(E^c\setminus B;\partial B)\leq C_1m'(t).
\end{equation}
Using again Equation \eqref{eq:volper},
\begin{equation}\label{eq:3}
P_K(E;B)=P_K(B\setminus E)-P_K(B\setminus E;\partial B)\geq P_K(B\setminus E)-m'(t).
\end{equation}
Since $|E_t|\geq |E|$ and $E$ is an isoperimetric region, Proposition \ref{lem:non} gives us
\begin{equation}\label{eq:4}
P_K(E)\leq P_K(E_t).
\end{equation}
Substituting \eqref{eq:2}, \eqref{eq:3} and \eqref{eq:4} in \eqref{eq:1}, it follows that
\[
P_K(B\setminus E)\leq (1+C_1)m'(t).
\]
From the Isoperimetric Inequality \eqref{eq:in},
\begin{equation}\label{eq:pre}
Cm(t)^{l-1/l}\leq m'(t).
\end{equation}
The above inequalities holds for a.e. $t>0$. Suppose that $m(t)>0$ for all $t\in[r/2,r]$, otherwise there is nothing to prove. Then we can rewrite Inequality \eqref{eq:pre} as
\[
C\leq \frac{m'(t)}{m(t)^{l-1/l}},
\]
and integrating between $r/2$ and $r$,
\[
r\leq C(m(r)^{1/l}-m(r/2)^{1/l})\leq Cm(r)^{1/l}\leq C\varepsilon ^{1/l}r.
\]
This is impossible for $\varepsilon$ small enough, and we get a contradiction. Therefore $|B(x,r/2)\setminus E|=0$.
\end{proof}

 \begin{remark}
	Reasoning as in , we could use $D$ in \eqref{in:doub} qq and, since the Poincaré inequality \eqref{des:Poincare} is global, we would obtain a isoperimetric inequality
	\begin{equation}
	C_2|F|^{D-1/D}\leq P_K(F)
	\end{equation}
	for all $F\subseteq \rr^d$ is any finite perimeter set and finite volume. Nevertheless, the lack of sharpness of $D$ in Lemma \ref{lem:ext} would not allow us to prove the following result.
\end{remark}

\begin{corollary}\label{cor:open} 
Let $E$ be an isoperimetric region. Then there exists an open set $E_0$ that coincides with $E$ almost everywhere.
\end{corollary}
\begin{remark}\label{rem:open}
	From now on we shall assume that an isoperimetric region $E$ is exactly $E_1$, and therefore an open set.
\end{remark}

The following Lemma allows us to deform a finite perimeter set $F$ increasing the volume so that the variation of the $K$-perimeter is controlled. For that, we shall take a cylinder $W^\lambda$ in $F^c$, which is located above the boundary of $F$ (see Figure 1). The upper cover shall be a translation by $(0,\ldots,0,s)$ of the lower cover, and thus have the same $K$-perimeter. The points in the lower cover are points of density $1$. Therefore the only part that is contributing to add perimeter is the contour of the cylinder, which can be bounded by the Riemannian one. On the other hand, the volume of $W^\lambda$ can  be computed by Fubini's Theorem.
\begin{lemma}[Deformation Lemma]\label{lem:def}
	Let $(\rr^d,\cdot,X,K)$ be a sub-Finsler nilpotent group. Let $F\subseteq \rr^d$ be a finite perimeter and finite volume set and suppose that there exists $p\in\text{int}(F)$. Then there exists $C_3>0$, $\lambda_0>0$, $\lambda_1>0$ and a family of finite perimeter sets $\{F^\lambda\}_{-\lambda_0<\lambda<\lambda_1}$ such that $|F^\lambda|=|F|+\lambda$ and
	\begin{equation}\label{eq:def}
	P_K(F^\lambda)-P_K(F)\leq C_3 |\lambda|.
	\end{equation}
	Moreover, if $F$ is bounded then $\lambda_1=+\infty$.
\end{lemma}	
\begin{proof}
	We set the notation $\phi(x)=(x_1,\ldots,x_{d-1})$,  $B(x,r)$ is the Euclidean ball of center $x$ and radius $r$, $D(x,r)=\phi(B(x,r))$, $C(x,r)=D(x,r)\times\rr$ and $v=(0,\ldots,0,1)$. Let $B(p,r)\subseteq F$. We claim that there exists $p'\in\text{int}(F^c)\cap C(p,r)$. Otherwise
	 \begin{equation*}
	 \begin{split}
	 |C(p,r)|&=|\text{int}(F)\cap C(p,r)|+|\partial F\cap C(p,r)|\\
	 &=|\text{int}(F)\cap C(p,r)|\leq|\text{int}(F)|=|F|\leq\infty,
	 \end{split}
 	\end{equation*}
	 which is a contradiction. Up to moving $p$ and taking a smaller $r$, we can assume that $p'-p\in\{(0,\ldots,0)\}\times\rr$ and denote $D=D(p',r/3)$.
	
		 
Let $L(x):=\{x-tv:0<t<t_0 \}$, where $x\in D\times \{p'_d\}$. The set $L(x)\cap\partial F$ is non-empty since $x\in F^c$ and $x-t_0 v\in F$. The boundary $\partial F$ is a closed set, and in  $L(x)$ is compact, so there exists a first point of intersection $p_x$. 

For any $0<\lambda$ we define $C^\lambda$ as
	$$W^\lambda:=\{p_x+tv: 0\leq t<\lambda /\mathfrak{L}^2(D),\ x\in D\},$$
	where $\mathfrak{L}^2$ is the two dimensional Lebesgue measure. Notice that the set $D\times(p'_d-\frac{\sqrt{3}}{2},p'_d+\frac{\sqrt{3}}{2})\subseteq B_{p'}(r/2)\subseteq F^c$, and let us denote by $v_M$ its volume.	
	Since $|W^\lambda|\leq v_M$, we can apply Fubini's Theorem to calculate the volume of $W^\lambda$ integrating the highs first, and we get that  $W^\lambda$ has volume $\lambda$.  
	
	\begin{figure}[h]\label{figure:1}
	\includegraphics[width=0.45\textwidth]{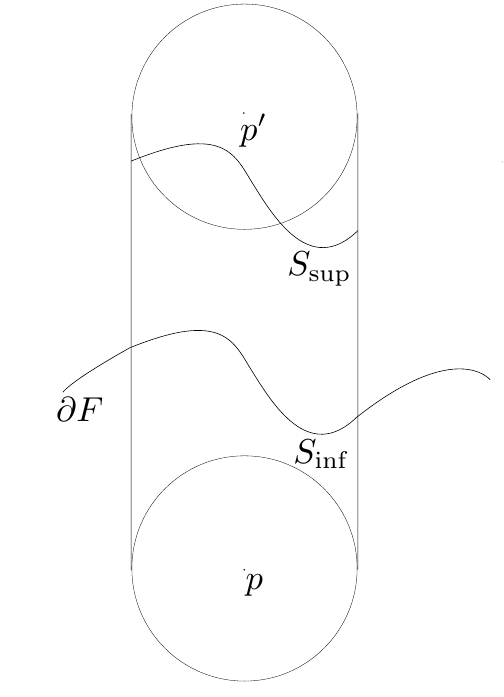}
		\caption{}
	\end{figure}

%
	Let $F^\lambda:=F\cup W^\lambda$. Then $W^\lambda\cap F=\emptyset$ and $|F^\lambda|=|F|+\lambda$.  We define the sets
	 \begin{equation*}
	 \begin{split}
	 &S_{\inf}:=\{p_x : x\in D \};\\
	 &S_{\sup}:=\Big\{p_x+\frac{\lambda}{|D|}: x\in D  \Big\};\\
	 &S_{\text{cont}}:=\Big\{ p_x+t : x\in \partial D; 0<t<\frac{\lambda}{|D|}  \Big\}.
	 \end{split}
	 \end{equation*}
	 It is clear that $\partial W^\lambda=S_{\inf}\cup S_{\sup}\cup S_{\text{cont}}$. We calculate the perimeter of $F^\lambda$: 
	\begin{equation}\label{p}
	\begin{split}
	P_K(F^\lambda)&=P_K(F^\lambda;\rr^d\setminus \bar{W^\lambda})+P_K(F^\lambda;\text{int}(W^{\lambda}))+P_K(F^\lambda;\partial W^{\lambda})\\
	 &=P_K(F;\rr^d\setminus \bar{W^\lambda})+P_K(F^\lambda;S_{\inf})+P_K(F^\lambda;S_{\sup})+P_K(F^\lambda;S_{\text{cont}}).
	\end{split}
	\end{equation} 
	For any $p\in S_{\inf}$, $p\in \text{int}(F^\lambda)$ and thus $P_K(F^\lambda;S_{\inf})=0$. Since the $K$-perimeter is invariant by vertical translations, 
	$$P_K(F;\bar{C}^\lambda)=P_K(F;S_{\inf})=P_K(W^\lambda;S_{\sup})=P_K(F^\lambda;S_{\sup}).$$
	Then
	\begin{equation}\label{eq:con}
	P_K(F^\lambda)=P_K(F)+P_K(F^\lambda;S_{\text{cont}}).
	\end{equation}
	From $\partial F^\lambda\cap S_{\text{cont}}= S_{\text{cont}}$, Inequality \eqref{eq:abscont} and the fact that the sub-Riemannian perimeter $P$ is less than the Riemannian one, $P_r$, we get
	 \begin{equation}\label{Pcont}
	\begin{split} 
	P_K(F^\lambda;S_{\text{cont}})&= P_{K}(W^\lambda;S_{\text{cont}})\\
	&\leq C_1 P(W^\lambda;S_{\text{cont}})\\
	&\leq C_1 P_{r}(W^\lambda;S_{\text{cont}})\\
	&=:C\lambda,
	\end{split}
	\end{equation}	
	where $P_{r}$ is the Riemannian perimeter. Substituting \eqref{Pcont} in \eqref{eq:con}, we get
	\[
	P_K(F^\lambda)-P_K(F)\leq C\lambda.
	\]
	
	
	This argument carries on subtracting a set $\tilde{W}^\lambda$ from $F$, so $|F^\lambda|=|F|-\lambda$ and $P_K(F^\lambda)-P_K(F)\leq C |\lambda|$.

Finally, if $F$ is bounded and $B(p,r)\subseteq F$, then we can find $B(p',r)\subseteq\text{int}(F^c)\cap C(p,r)$ and apply the previous result, so that $\lambda_1\geq v_M$. We apply again the proven to the set $F^{v_M}$ taking the same ball $B(p,r)$ and, iterating, we get the result.
\end{proof}	
\begin{remark}\label{remark1}
	The constant $C_3$ depends on the radius of the ball $B(p,r)$ taken inside $F$. Therefore, if $r>0$ and $F_1$ and $F_2$ are two finite perimeter sets with finite volume such that $B(p_1,r)\subseteq F_1$ and $B(p_2,r)\subseteq F_2$ for some $p_1$ and $p_2$, then we can take $C_3>0$  satisfying \eqref{eq:def}.
\end{remark}	
\begin{remark}
	To prove Lemma \ref{lem:def}, we could think of taking a horizontal vector field $U$ and use the formulas for the first variation of the volume and  the area to bound from above the extra perimeter of the deformation by the Riemannian one and use the deformation Lemma in $\rr^d$. Nevertheless, if the Riemannian perimeter of $F\cap\text{supp}(U)$ is infinite, the upper bound would be infinite.
\end{remark}	
In Lemma  \ref{lem:ext} we proved that if we have a ball that is almost in an isoperimetric region $E$, then the ball of half the radius is in $E$.
Following again the arguments in \cite{MR2000099} and using the Deformation Lemma \ref{lem:def}, we prove in Lemma \ref{lem:int} the analog result when the starting ball is almost outside $E$.
\begin{lemma}\label{lem:int}
		Let $E$ be an isoperimetric region, $x\in\rr^d$ and $0<r\leq1$. Then there exists $\varepsilon>0$ such that  if $r^{-l}|E\cap B(x,r)|\leq \varepsilon$, then
		\[
			|E\cap B(x,r/2)|=0.
		\]	
	where is $B(x,r)$ the sub-Riemannian ball centered in $x$ of radius $r>0$.
\end{lemma}
\begin{proof}
Let  $t>0$, $B:=B(x,t)$ and $m(t)=|E\cap B|$. Under the assumption that $r^{-l}|E\cap B(x,r)|\leq \varepsilon$, $m(t)$ is small enough to use the Deformation Lemma \ref{lem:def}, and define $E_t:=(E\setminus B)^{m(t)}$. It is clear that $|E_t|=|E|$. Thus
\begin{equation*}
P_K(E)\leq P_K(E_t).
\end{equation*}
On the other hand, reasoning as in Lemma \ref{lem:ext} and using the Deformation Lemma, we get
\begin{equation*}
\begin{split}
P_K(E_t)&\leq P_K(E\setminus B)+C_3m(t)\\
&\leq P_K(E)-P_K(E\cap B)+(1+C_1)m'(t)+C_3 m(t)\\
&\leq P_K(E)-m(t)^{l-1/l}+(1+C_1)m'(t)+C_3m(t).
\end{split}
\end{equation*}
The above inequalities gives us
\begin{equation*}
m(t)^{l-1/l}-(1+C_1)m(t)\leq C_3m'(t).
\end{equation*}
For $m(t)$ small enough, there exists $C>0$ such that 
$$
Cm(t)^{l-1/l}\leq m(t)^{l-1/l}-(1+C_1)m(t),$$
and
\begin{equation*}
C'm(t)^{l-1/l}\leq m'(t).
\end{equation*}
Again, supposing that $m(t)>0$ for all $t\in [r/2,r]$, we have
\[
C'\leq \frac{m'(t)}{m(t)^{l-1/l}},
\]
and integrating over $r/2$ and $r$, we get a contradiction for $\varepsilon>0$ small enough.
\end{proof}

Let $E$ be an isoperimetric region, we define the sets
\begin{equation*}
\begin{split}	
E_1&=\{x\in\rr^d : \exists r>0 \ \text{such that }|B(x,r)\setminus E|=0 \}\\
E_0&=\{x\in\rr^d : \exists r>0 \ \text{such that }|E\cap B(x,r)|=0\}\\
S&=\{x\in\rr^d : h(x,r)>\varepsilon \ \forall r\leq 1 \}.
\end{split}
\end{equation*}

\begin{theorem}\label{thm:boundary}
	Let $(\rr^d,\cdot,X,K)$ be a sub-Finsler nilpotent group and let $E$ be a isoperimetric region. Then the topological and essential boundaries of $E$ coincide.
\end{theorem}

\begin{proof}
 By Lemma \ref{lem:ext} and \ref{lem:int}, the sets $E_0$, $E_1$ and $S$ form a partition of $\rr^d$. Since $E_1$ and $E_0$ are open and disjoint, $\partial E_1\cup \partial E_0\subseteq S$. On the other hand, if $x\in S$ and $r>0$, $B(x,r)\cap E_1\neq\emptyset$ then $B(x,r)\cap E_0\neq\emptyset$, otherwise $x\in \text{int}(E_1)$, and $x\in\partial E_1\cap\partial E_0$.
\end{proof}

\begin{theorem}[Boundedness]\label{thm:bound}
	Any isoperimetric region in a sub-Finsler nilpotent group $(\rr^d,\cdot,X,K)$ is bounded.
\end{theorem}	

\begin{proof}
	Let $E$ be an isoperimetric set of volume $v$, $B$ the sub-Riemannian ball centered in $0$ of radius $r>0$, and $m(r)=|E\setminus B|$, and $(E\cap B)^{m(r)}$ be the family given by Lemma \ref{lem:def}. Since $|(E\cap B)^{m(r)}|=v$, we have
	\begin{equation}\label{eq:min}
	P_K(E)\leq P_K((E\cap B)^{m(r)}).
	\end{equation}
	Using the Deformation Lemma \eqref{eq:def}, the Isoperimetric inequality \eqref{eq:in} and Equations \eqref{eq:volper} and \eqref{des:decom}, we get
	\begin{equation}\label{eq:bound}
	\begin{split}
	P_K((E\cap B)^{m(r)})\leq& P_K(E\cap B)+ C_3m(r)\\
	\leq& P_K(E)-P_K(E\setminus B)+2C_1P(E\cap B;\partial B)+C_3m(r)\\
	\leq& P_K(E)-P_K(E\setminus B)-2C_1m'(r)+C_3m(r)\\
	\leq&  P_K(E)-C_2m(r)^{l-1/l}-2C_1m'(r)+C_3m(r).
	\end{split}
	\end{equation}	
	Subtracting Inequality \eqref{eq:min} in \eqref{eq:bound}, we get
	\[
	-2C_1m'(r)\geq C_2m(r)^{l-1/l}-C_3m(r).
	\]
	As $m(r)$ tends to $0$ as $r$ tends to $\infty$, for $r$ big enough there exists $C>0$ such that
	\[
	C_2m(r)^{l-1/l}-C_3m(r)\geq C m(r)^{l-1/l}.
	\]
	Let $r>1$ and suppose that $m(r)>0$. Then
	\begin{equation}\label{eq:bo}
	\frac{-m'(r)}{m(r)^{l-1/l}}\geq C,
	\end{equation}
	Since 
	\begin{equation}\label{eq:bou}
	\int_{1}^{r}\frac{-m'(s)}{m(s)^{l-1/l}}=-\int_{m(1)}^{m(r)}\frac{1}{s^{l-1/l}}ds=m(1)^{1/l}-m(r)^{1/l}.
	\end{equation}
	Integrating \eqref{eq:bo} between $1$ and $r$ and using \eqref{eq:bou}, we get
	\[
	m(1)^{1/l}\geq Cr+m(r)^{1/l}
	\]
	and $r$ is bounded.
\end{proof}

\section{Existence of isoperimetric regions}
Throughout this section,  $K$ shall denote a convex body in $\hhh^0_0$ containing $0$ in its interior and $B(x,r)$ the sub-Riemannian ball centered in $x$ of radius $r>0$.  We shall follow the arguments of Galli and Ritoré \cite{MR2979606}.  

The following lemma can be found in \cite{MR2000099} for Carnot groups, and in the context of sub-Finsler nilpotent groups the proof can be done \textit{mutatis mutandis}.

\begin{lemma}[Concentration Lemma]\label{lem:con}
	Let $F$ be a set with finite perimeter and volume. Suppose that there exists $m\in(0,|B(0,1)|/2)$ such that $|F\cap B(x,1)|<m$ for all $x\in \rr^d$. Then there exists $C>0$ depending only on $l$ such that
	\[
	C|F|^l P_K(F)^{-l}\leq m.
	\]
\end{lemma}	
%
%

The following Lemma can be found in \cite{leonardi-ritore}.
\begin{lemma}\label{lem:seq}
	Let $\{E_n\}$ be a sequence of uniformly bounded perimeter sets of volumes $\{v_n\}$ converging to $v>0$. Let $E$ be the limit in $\L^1_{loc}(\rr^d)$ of $E_n$. Then there exists a divergence sequence of radii $\{r_n\}$ such that, setting $F_n=E_n\setminus B(0,r_n)$ and up to a subsequence, it is satisfied
	\begin{equation}
	\begin{split}
	|E|+&\liminf_{n\to\infty}|F_n|=v,\\
	P_K(E)+&\liminf_{n\to\infty}P_K(F_n)\leq \liminf_{n\to\infty}P_K(E_n).
	\end{split}
	\end{equation}
\end{lemma}	

\begin{proof}
	Take $\{s_n\}$ increasing with $s_n-s_{n+1}>n$. We claim that there exists $r_n$ in $[s_n,s_{n+1}]$ such that 
	$$P(E_n\cap\partial B(0,r_n);\partial B(0,r_n))<v_n/n,$$
	 where $P$ and $B(0,r)$ are the sub-Riemannian perimeter and ball of center $0$ and radius $r$ respectively. Otherwise, by Inequality \eqref{eq:volper} we have
	\begin{equation*}
	v_n<\int_{s_n}^{s_{n+1}}P(E_n\cap\partial B(0,t))dt\leq \int_{s_n}^{s_{n+1}} \frac{d}{ds}|_{s=t}|E_n\cap B(0,s)| dt\leq v_n.
	\end{equation*}
	Therefore, by Inequality \eqref{des:decom} we get
	\begin{multline}\label{eq:pe}
	P_K(E_n)\\
	\begin{split}
	&\geq P_K(E_n\cap B(0,r_n))+P_K(E_n\setminus B(0,r_n))-2C_1P(E_n\cap \partial B(0,r_n);\partial B(0,r_n))\\
	&\geq P_K(E_n\cap B(0,r_n))+P_K(F_n)-2C_1v_n/n.
	\end{split}
	\end{multline}
	On the other hand,
	\begin{equation}\label{eq:vo}
	|E_n|=|E_n\cap B(0,r_n)|+|E_n\setminus B(0,r_n)|.
	\end{equation}
	Taking inferior limits in $n$ in \eqref{eq:pe} and \eqref{eq:vo}, and using the lower semicontinuity, we have the result.
\end{proof}	

\begin{theorem}[Existence of isoperimetric regions]\label{thm:existence}
	Let $(\rr^d,\cdot,X,K)$ be a sub-Finsler nilpotent group. Then, for any $v>0$, there exists a finite perimeter set $E$ such that $|E|=v$ and $I_K(v)=P_K(E)$. Moreover, $E$ has a finite number of connected components.
\end{theorem}
\begin{proof}
	Let $\{E_n\}_{k\in\mathbb{N}}$ be a minimizing sequence of sets with $|E_n|=v$ and $P_K(E_n)\leq I_K(v)+\frac{1}{n}$. By compactness, the sequence converges in $\L^1_{loc}(\rr^d)$ to a set $E^0$. Let $v_0:=|E^0|$.  By Lemma \ref{lem:seq}, we can find a sequence of divergence radii $r_n$ such that, denoting $F_n:=E_n\setminus B(0,r_n)$, we have
	\begin{equation}\label{eq:per}
	\begin{split}
	v_0+&\liminf_{n\to\infty}|F_n|=v,\\
	P_K(E^0)+&\liminf_{n\to\infty}P_K(F_n)\leq I_K(v).
	\end{split}
	\end{equation}
	 If $v_0=v$, then the Theorem is proven.  If $v_0<v$, we claim that $E_0$ is isoperimetric for its volume. Otherwise, we can find $O\subseteq G$ such that $|O|=v_0$ and $P_K(O)< P_K(E^0)$. By Theorem \ref{eq:bound}, $O$ is bounded and by definition of $F_n$, we can find $n_0$ such that $\forall n>n_0$, $O$ and $F_{n}$ are disjoint. Then 
	$$\liminf_{n\to\infty}|O\cup F_n|=|O|+\liminf_{n\to\infty}|F_n|=v.$$
	 By Equation \eqref{eq:per}, 
	\begin{equation*}
	\begin{split}
	I_K(v)&\leq \liminf_{n\to\infty}P_K(O\cup F_{n})\\
	&=P_K(O)+\liminf_{n\to\infty}P_K(F_{n})\\
	&<P_K(E^0)+\liminf_{n\to\infty}P_K(F_{n})\\
	&\leq I_K(v),
	\end{split}
	\end{equation*}
	and we have a contradiction.
	
	Step two. If $v_0<v$ we apply Lemma \ref{lem:con} to find a divergent sequence of points $\{x_n^1 \}$ such that $|F_n\cap B(x_n,1)|\geq m_0|F_n|$. The sets $\ell_{-x_n} F_n$ converge in $\L^1_{loc}(\rr^d)$ to a set $E^1$ of volume $v_1\leq\lim_{n\to\infty}|F_n|= v- v_0$. By Lemma \ref{lem:seq}, we can find a divergent sequence $\{r'_n \}$ of radii so that the sets $F'_n:=(\ell_{-x_n} F_n)\setminus B(0,r'_n)$ verifies
	\begin{equation}
	\begin{split}
	v_1+&\liminf_{n\to\infty}|F'_n|=v-v_0,\\
	P_K(E^1)+&\liminf_{n\to\infty}P_K(F'_n)\leq \liminf_{n\to\infty}P_K(F_n).
	\end{split}
	\end{equation}
	Since $E^0$ in bounded, we can suppose that $E^0\cap E^1=\emptyset$. If $v_1=v-v_0$, then $|E^0\cup E^1|=|E^0|+|E^1|=v$ and 
	$$P_K(E^0\cup E^1)=P_K(E^0)+P_K(E^1)\leq P_K(E^0)+\lim_{n\to\infty} P_K(F_n)\leq I(v).$$
	Thus $E^0\cup E^1$ is the isoperimetric region of volume $v$. If $v_1<v-v_0$, then $E^0\cup E^1$ is isoperimetric for its volume. Otherwise there exists $O\subset G$ such that $|O|=v_0+v_1$ and $P_K(O)<P_K(E^0)+P_K(E^1)$. Then
	\begin{equation*}
	\begin{split}
		I_K(v)&\leq \liminf_{n\to\infty}P_K(O\cup F'_{n})\\
		&=P_K(O)+\liminf_{n\to\infty}P_K(F'_{n})\\
		&<P_K(E^0)+P_K(E^1)+\liminf_{n\to\infty}P_K(F'_{n})\\
		&\leq P_K(E^0)+\liminf_{n\to\infty}P_K(F_n)\\
		&\leq I_K(v),
		\end{split}
	\end{equation*}
	and we have a contradiction. 
%
	
	By induction, we get a sequence of sets $E^0,\ldots,E^n$ pairwise disjoint of volumes $v_0,\ldots,v_n$ whose union is isoperimetric for its volume $\sum_{i=1}^n v_i$. Suppose that there exists a infinite number of pieces $E^i$. Then $\sum_{i=0}^\infty v_i\leq v$. Let $j$ and $k$ with $v_j\geq v_i$ for all $i$ and $v_k$ small enough so  we can take $(E^j)^{v_k}$ the family defined in \ref{lem:def} and there exists $C>0$ with $Cv_k^{l-1/l}>C_3 v_k$. Then, by the  Deformation Lemma  \eqref{eq:def} and the Isoperimetric Inequality \eqref{thm:in}, we get
	\begin{multline*}
	I_K(\sum_{i} v_i)\leq\sum_{i\neq j,k}P_K(E^i)+P_K((E^j)^{v_k})\leq\sum_{i\neq k} P_K(E^i)+C_3 v_k\\
	<\sum_{i\neq k} P_K(E^i)+Cv_k^{l-1/l}\leq\sum_{i\neq k} P_K(E^i)+P_K(E^k)=I_K(\sum_i v_i),
	\end{multline*}
	which is a contradiction. Therefore there are a finite number of pieces, $r$, until $\sum_{i=1}^r v_i\geq v$, and $E^0\cup E^1 \ldots\cup E^r$ is the isoperimetric region of volume $v$.
\end{proof}

\begin{corollary}\label{cor:sub-add}
	Let $(\rr^d,\cdot,X,K)$ be a sub-Finsler nilpotent group. The isoperimetric profile $I_K$ is sub-additive.
\end{corollary}

\begin{proof}
		Let $v_0,\ldots, v_n\geq 0$. Take $E_k$ isoperimetric region of volume $v_k$. By Theorem \ref{thm:bound}, $E_k$ is bounded and we can take $E_j\cap E_i=\emptyset$. Therefore
 		\[
		I_K(v_0+\ldots+v_n)\leq P_K(\bigcup_{k=1}^n E_k)=\sum_{k=1}^n P_K(E_k)=\sum_{k=1}^n I_K(v_k).\qedhere
		\]	
\end{proof}		
	
The locally Lipschitz property of the isoperimetric profile follows from the existence and the following uniform version of the Deformation Lemma.

\begin{lemma}\label{lem:unifdef}
Let $0<w<v$  and denote by $E_x$ an isoperimetric region of volume $w\leq x\leq v$. There exist $C_3>0$, $\lambda_1>0$ and a family of sets $\{E_x^\lambda: -\lambda_1<\lambda, w\leq x\leq v \}$ such that $|E_x^\lambda|=x+\lambda$ and
\begin{equation}\label{eq:unifdef}
P_K(E_x^\lambda)\leq P_K(E_x)+C_3\lambda.
\end{equation}
\end{lemma}
\begin{proof}
	By Remark \ref{remark1}, It is enough to show that there exist $r_0>0$ and $\{p_x\}_{w\leq x\leq v}$ such that $B(p_x,r_0)\subseteq E_x$ for all $w\leq x\leq v$. Let 
	$$r_x:=\sup\{r>0 :B(q_x,r)\subseteq E_x \text{ for some } q_x\in E_x \}.$$
	We shall prove that $\inf_{x\in [w,v]} r_x>0$. 
	Otherwise, there is a sequence $\{x_n\}\subset [w,v]$ such that $r_{x_n}<\frac{1}{n}$. 
	
  We claim that there exists $m>0$ and $p'_x$ such that $|E_{x}\cap B(p'_x,1)|\geq m$ for all $w\leq x\leq v$. Otherwise, By Lemma \ref{lem:con} and Proposition \ref{lem:non},
  \[
	 C^l|E_w|\leq C^l |E_x|\leq \frac{1}{n} P_K(E_x)\leq \frac{1}{n} P_K(E_v).
  \]		
	
 By abuse of notation, we denote as $E_{x_n}$ the set $\ell_{p'_{x_n}}E_{x_n}$. By Proposition \ref{lem:non}, $P_K(E_{x_n})\leq P_K(E_v)$ and, by compactness, there exists a finite perimeter set $F$ such that $E_{x_n}\to F$ in $L^1_{loc}(\rr^d)$. Therefore $|E_{x_n}\cap B(0,1)|\geq m>0$ and by $L^1_{loc}$-convergence, $|F\cap B(0,1)|\geq m>0$.
 
  Finally we will prove that $|F|=0$ to get a contradiction.  Suppose that there exists a point $p$ where the density of $F$ at $p$ is $1$. Then there exists $r>0$ such that
  \[
  2\varepsilon>\frac{|B(p,r)\setminus F |}{|B(p,r)|}\geq Cr^{-l}|B(p,r)\setminus F|.
  \]
  By $L^1_{loc}(\rr^d)$-convergence, $r^{-l}|B(p,r)\setminus E_{x_n}|<\varepsilon$ and by Lemma \ref{lem:ext} and Remark \ref{rem:open},  $B(p,r/2)\subseteq E_{x_n}$, which is a contradiction with $\lim_{n\to \infty}r_{x_n}=0$. 
\end{proof}

\begin{corollary}\label{prop:cont}
	Let $(\rr^d,\cdot,X,K)$ be a sub-Finsler nilpotent group. The isoperimetric profile is a locally Lipschitz function.
\end{corollary}	
\begin{proof}
	Let $0<w_0<v$,  $E_v$ isoperimetric region of volume $v$ and $\{E_w^\lambda : -\lambda_1<\lambda, w\in [w,v]\}$ the family defined in Lemma \ref{lem:unifdef}. Up to taking $w_0$ closer to $v$, we can assume that $v-w_0<\lambda_1$. Notice that $|E_v^{w-v}|=w$. By \eqref{eq:unifdef}, we get
\begin{equation*}
\begin{split}
I_K(v)&=P_K(E_v)\geq P_K(E_v^{w-v})-C_3(v-w)\geq I_K(w)-C_3(v-w) \\
I_K(v)&\leq P_K(E_w^{v-w})\leq P_K(E_w)+C_3(v-w)=I_K(w)+C_3(v-w).
\end{split}
\end{equation*}	
By a similar reasoning for $0<v<w_0$, we get that $I_K$ is locally Lipschitz.
\end{proof}



\end{document}